\documentclass[12pt]{amsart}
\usepackage{latexsym,fancyhdr,amssymb,color,amsmath,amsthm,graphicx,listings,comment}
\newtheorem{thm}{Theorem}       
\newtheorem{lemma}{Lemma}       \newtheorem{coro}{Corollary}
\usepackage[section]{placeins}      \let\paragraph\subsection
\title{A Reeb sphere theorem in graph theory} \author{Oliver Knill} \date{3/24/2019}
\address{Department of Mathematics \\ Harvard University \\ Cambridge, MA, 02138 }

\begin{document}

\begin{abstract}
We prove a Reeb sphere theorem for finite simple graphs. 
The result bridges two different definitions of spheres in graph theory.
We also reformulate Morse conditions in terms of the center manifolds, 
the level surface graphs $\{ f= f(x) \}$ in the unit sphere $S(x)$.  
In the Morse case these graphs are either spheres, the empty graph or 
the product of two spheres.
\end{abstract}
\maketitle

\section{Introduction}

\paragraph{}
It was Herman Weyl \cite{Weyl1925} who first asked to look for a combinatorial finite 
analogue of a $d$-dimensional Euclidean sphere: how can one, in a finite manner, without 
assuming the existence of infinity, capture results in differential topology?
Any referral to the continuum like using triangulations is not an option because it builds on the 
Euclidean notion of space. Spheres are important because spheres allow to compensate for the 
absence of tangent spaces as the unit sphere of a point plays the role of the tangent space.
Using sphere geometry one can model subspaces of the tangent spaces as
intersections of spheres. 

\paragraph{}
With a notion of sphere, one has so a notion of a locally Euclidean space
and can do geometry on finite set of sets mirroring results almost verbatim
from the continuum. We will use functions $f$ on graphs which are real valued 
but it should be clear from the discussion that one could use functions which 
take values in a finite set like $\{1, \dots, n\}$ to make sure that the 
infinity axiom is never used. We insist in particular to avoid any
geometric realizations to prove things. 

\paragraph{}
In the 1990ies, two approaches have emerged which answer the question of Weyl: 
``digital topology" spearheaded by Alexander Evako uses an induction on 
spheres and Whitehead homotopy (reformulated purely combinatorially) 
and defines a sphere as a structure in which every unit sphere is a by 1 
lower dimensional sphere and where puncturing the sphere renders the object 
contractible (see e.g. \cite{Evako2013}).
A second, more analytic approach is ``discrete Morse
theory" was pioneered by Robin Forman \cite{forman95,Forman2002}. 
It deals with functions and critical points. In that frame work, 
the punctured collapse is replaced by the existence of a function with exactly two critical points.
This approach is well suited to model discrete analogues of smooth situations and
gives a Morse theory which is equivalent to the continuum. The theory still refers
often to the continuum although. 

\paragraph{}
While mending the two approaches is maybe a bit ``obvious", it is still necessary to have 
this formulated precisely. The exercise can be solved comfortably in the language of 
graphs. There is not much of a loss of generality when using graph theory rather than 
working with simplicial complexes, because any simplicial complex 
$G$ has a Barycentric refinement $G'$ that is the Whitney complex of a graph. 
A finite abstract simplicial complex is then declared to be a $d$-sphere if and only if its 
Barycentric refinement is a $d$-sphere as a graph. 
Equipped with more general simplicial complexes, graphs become powerful structures to work with. 
Their advantage is that they are more intuitive than finite sets of sets. 
It is quite astonishing what kind of mathematics one can reach with this 
language. Some is summarized in \cite{AmazingWorld}.

\paragraph{}
For the author of this note, who started to explore geometric topics in graph theory in
\cite{elemente11,cherngaussbonnet}, the analytic approach 
first appeared to be more natural, especially in the context of the 
discrete Poincar\'e-Hopf theorem \cite{poincarehopf}, which works for a general graph and where a vertex $x$
is a critical point if $S_f^-(x)$, the graph generated by $\{ y \in S(x) \; | \; f(y)<f(x) \}$ 
is not contractible and where the index is $i(x)=1-\chi(S_f^{-}(x))$ with Euler characteristic $\chi$. 
In that frame work, spheres can be defined as the class of graphs which admit a function 
with exactly two critical points, while contractible graphs are characterized allowing
functions with one critical points. Later, when looking at coloring problems, the homotopy definition 
became attractive. The inductive definition is well tailored for proofs like the 
Jordan-Brower-Schoenfliess theorem \cite{KnillJordan}. Simple facts like that the 
Euler characteristic of a $d$-sphere $S^d$ is $\chi(S^d) = 1+(-1)^d$ have there straightforward
proofs because the definitions allow for short recursive arguments.

\paragraph{}
There is an other reason to write down the current note:
when coloring graphs using topology \cite{knillgraphcoloring,knillgraphcoloring2,knillgraphcoloring3}, 
it is necessary to build up a $3$-ball as an increasing sequence of Eulerian $3$-balls $G_n$, 
where it is pivotal that in every step, the boundary $G_n$ is a $2$-sphere and where during the 
build-up, no topology changes happen. While programming this on a computer, 
we needed to do refinements which were at first not expected. The current paper illustrates
a bit the difficulty when working with chosen inductive definitions. As they are detached
from any Euclidean embedding, there are things which are overlooked at first. The actual 
aim then is to color planar graphs of $n$ vertices constructively and deterministically in $O(n)$ 
time possibly with explicit constants for the $O(n)$ part. 

\paragraph{}
While working with the recursive definition of spheres is more
elegant for general proofs, we need to think in Morse pictures when building algorithms. We will see in 
particular that the Reeb sphere theorem shows that any $3$-ball can be ``foliated" into $2$-balls 
(where neighboring leaves can intersect). This slicing requires the existence of a function $f$ on the 
$d$-ball which has exactly two critical points, a maximum and minimum. 
This foliation is not possible in the graph itself.
But fortunately, the notion of level surfaces $\{ f=c \}$ in a $d$-graph 
is nice in the discrete because there are never singularities if $c$ is not in the range $f(V)$ of $f$: 
the level surface is always a $(d-1)$-graph \cite{KnillSard}, so that only vertices can be critical points.
This mirrors the classical Sard theorem, where for almost all $c$ in the target space of a smooth function $f$
on a manifold $M$, the inverse $f^{-1}(c)$ is either empty or a discrete manifold. 

\paragraph{}
All graphs $G$ considered here are finite simple and all functions $f$ are {\bf colorings}, meaning that
they are locally injective functions on the vertex set $V$ of the graph. This especially implies
that the value $c=f(x)$ is not in the range $f(S(x))$ for any 
unit sphere $S(x)$. A $d$-graph is now a graph for which every unit sphere $S(x)$ is a $(d-1)$-sphere and 
a $d$-sphere is a $d$-graph which when punctured becomes contractible. A graph is defined to be contractible
if there exists $x$ such that $G-x$ is contractible. The inductive definitions are booted up with the 
assumption that the empty graph is the $(-1)$-sphere and the $1$-point graph $K_1$ is the $0$-ball. 
A graph is declared to be a {\bf $d$-ball} if it can be written as a punctured $d$-sphere. These
definitions allow to switch quickly from spheres to balls, by puncturing a sphere $G$ to get a ball 
$G-x$ or to build a cone extension over the boundary of a ball to get a sphere. 

\paragraph{}
In this paper, we declare a vertex $x$ in a $d$-graph $G=(V,E)$ to be a 
{\bf critical point} of $f: V \to \mathbb{R}$ if at least one of the sets 
$S_f^-(x) = \{ y \; | \; f(y) < f(x) \}$ or $S_f^+(x)=\{ y \; | \; f(y)>f(x) \}$ is not contractible.
This is a slightly more inclusive definition than asking that $S_f^-(x)$ is not contractible. It leads
to more critical points. In order not to be ambiguous, we might call points for which $S_f^-(x)$ is not
contractible, a {\bf Poincar\'e-Hopf critical point} or {\bf one sided critical point}. For $d$-graphs,
the two definitions will agree. The main reason to consider the symmetric modification 
is the situation of $d$-graphs with boundary: a $d$-ball admits a function with exactly $2$ critical 
points (the maximum and minimum). The one-sided Poincar\'e-Hopf definition allows for a function 
on a $d$-ball with exactly one critical point, the minimum (by definition as a $d$-ball is defined 
to be contractible). 

\paragraph{}
We will just see that for $d$-graphs, the property that $S_f^+(x)$ is contractible is 
equivalent to that both $S_f^{\pm}(x)$ are both $(d-1)$-balls and that the center manifold
$B_f(x)=\{ f = f(x) \}$ is a $(d-2)$-sphere. So, for $d$-graphs which by definition 
have no boundary, the contractibility of both the stable and unstable sphere
is equivalent to the one-sided definition in which only the stable sphere $S^-_f$ is required to 
be contractible. Already for $d$-balls, the two definitions are no more equivalent: 
a $d$-ball has at least two critical points (the maximum and minimum) 
in the above sense. In the one sided definition however there is 
a function $f$ with exactly one critical point, the minimum. 
The combinatorial Reeb statements formulated here are not deep. 
They certainly do not have the subtlety of the {\bf Jordan-Brower-Schoenfliess theorem}
which covers an opposite angle to the story and which requires to look at global properties.
We appear to need Jordan-Brower-Schoenfliess when characterizing center manifolds of regular
points however. 

\paragraph{}
The discrete Jordan-Brower-Schoenfliess theorem \cite{KnillJordan} 
assures that a $(d-1)$ sphere $H$ in a $d$-sphere $G$ divides $G$ into two parts which 
are both balls. This is a reverse of the statement that if $G$ is partitioned into two 
balls, the interface is a sphere. The later is harder to prove than the discrete Reeb result. 
In Reeb, we assume contractibility on both sides and deduce from this that they must be balls 
with a common sphere boundary. The Jordan-Brower-Schoenfliess is the reverse: it starts with 
an embedded $(d-1)$-sphere in a $d$-sphere and concludes that this splits the sphere into 
two disjoint balls which in particular are contractible. The {\bf Alexander horned sphere} 
illustrates that the classical topological case can be tricky: there is a
topological $2$-sphere embedded in a $3$-sphere, such that only one side is simply connected. 
The discrete case is more like the piecewise linear or smooth situation in the continuum. 

\paragraph{}
We can reformulate the above given notion of critical point also using the 
{\bf center sub manifold} $B_f(x) = \{f = f(x) \}$ of $S(x)$ which is always a
$(d-2)$-graph by the local injectivity assumption and discrete Sard \cite{KnillSard}.
A vertex $x$ is now a critical point if and only if $B_f(x)$ is not a sphere. 
The center manifold $B_f(x)$ allows to comfortably define what a 
discrete analogue of a {\bf Morse function} is: a function on a $d$-graph is called
a {\bf Morse function}, if every $B_f(x)$ is either a sphere or a Cartesian product of two spheres
or then is empty. (The empty case belongs also to the cases of critical points. If you like, it is the
product of two spheres too, where one is the empty graph and the other $S(x)$. We prefer however to leave the
product with a graph with the empty graph to be undefined.)
The former case happens for {\bf regular points} while the later is the situation at a
{\bf critical point}. In the later case, the sphere itself $S(x)$ is the {\bf join} of two spheres
and has dimension $d-1$ while the Cartesian product $B_f(x)$ has dimension $d-2$ or is empty. 

\paragraph{}
The center manifold $B_f(x)$ is useful in general for any finite simple graph but it is especially appealing
when looking at curvature \cite{indexformula}. 
Curvature is an expectation of Poincar\'e-Hopf indices \cite{indexexpectation}
where the probability space is a set of functions like the finite set of all colorings with a minimal
number of colors. Since we can write such an index in terms of
an Euler characteristic of a center manifold, the Euler characteristic is expressible as an average of 
Euler characteristis of co-dimension two surfaces. By Gauss-Bonnet,
Euler characteristic can then be written as an expectation of curvatures 
of such surfaces. 

\paragraph{}
The fact that for even-dimensional $d$-graphs, 
the Euler characteristic is an expectation of curvatures of two-dimensional graphs, whose curvature can be
interpreted as {\bf sectional curvatures} brings Euler characteristic close to the Hilbert action in 
relativity. I shows that one can see Euler characteristic as a
{\bf quantized version} of the Hilbert action \cite{KnillFunctional}.
Now, it depends how the variational problem is 
set up in the discrete but one can see $d$-graphs as critical points of 
Euler characteristic: doing a variation like removing
a point does not change the Euler characteristic or makes it smaller for even dimensional 
manifolds and larger for odd dimensional manifolds. 

\paragraph{}
The functional $\chi$ in some sense can explain why we observe space which appears Euclidean. 
A typical random disorganized network is not a critical point of the Euler characteristic. 
There are other arguments for Euclidean structures like that the Barycentric refinement 
operation grows Euclidean structure \cite{KnillBarycentric2}: each simplex becomes a ball and becomes
so part of Euclidean space already after one Barycentric refinement. For $d$-graphs,
the functional $\chi(G) = \sum_x \omega(x)$ 
agrees with Wu characteristic $\omega(G) = \sum_{x \sim y} \omega(x) \omega(y)$ 
\cite{Wu1953,valuation,CohomologyWuCharacteristic}, which measures the pair-interaction of 
intersecting ($\sim$) simplices $x,y$ in a simplicial complex $G$. For $d$-graphs with boundary we have
$\omega(G) = \chi(G)-\chi(\delta G)$, where $\delta G$ is the boundary $(d-1)$ complex.

\begin{figure}[!htpb]
\scalebox{0.57}{\includegraphics{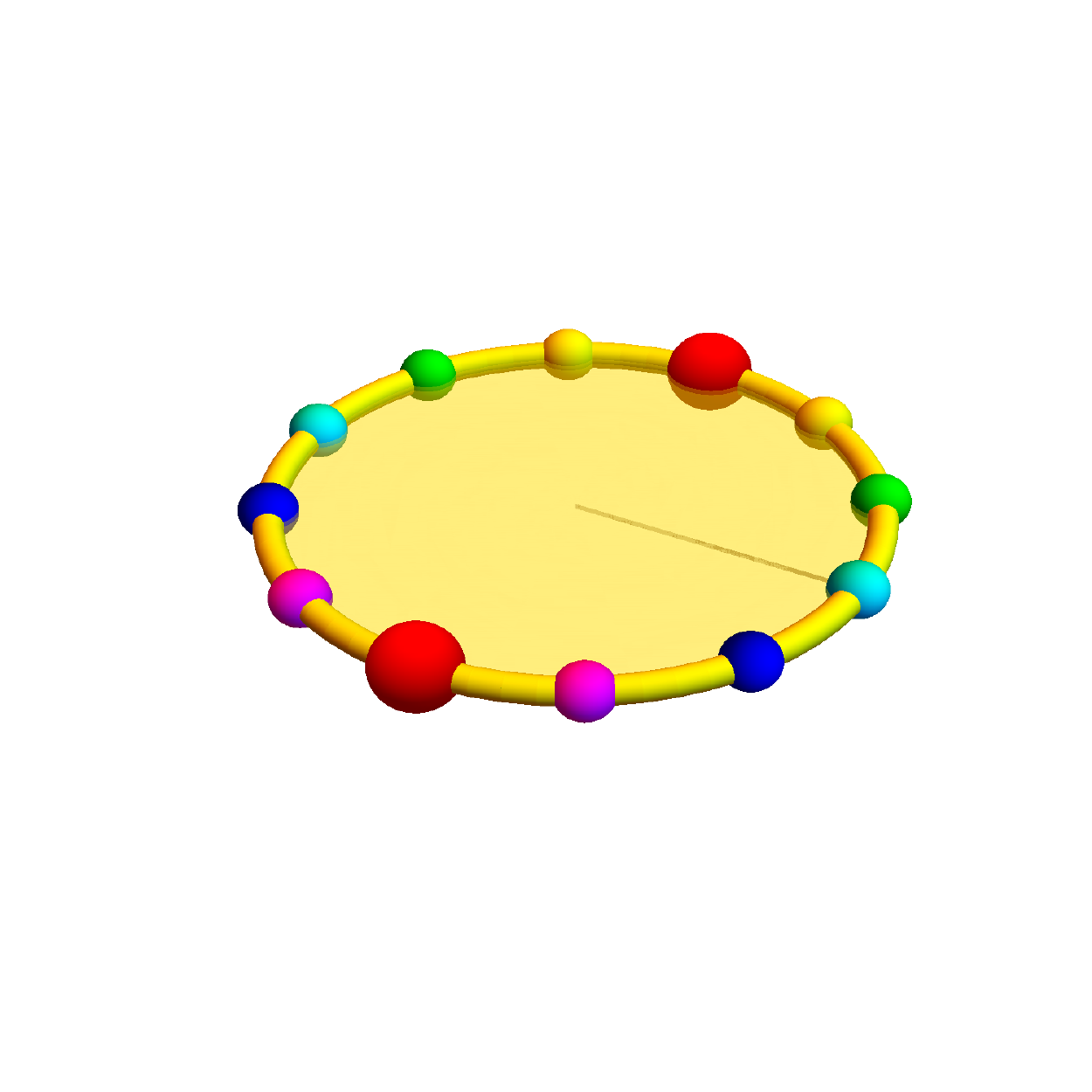}}
\scalebox{0.57}{\includegraphics{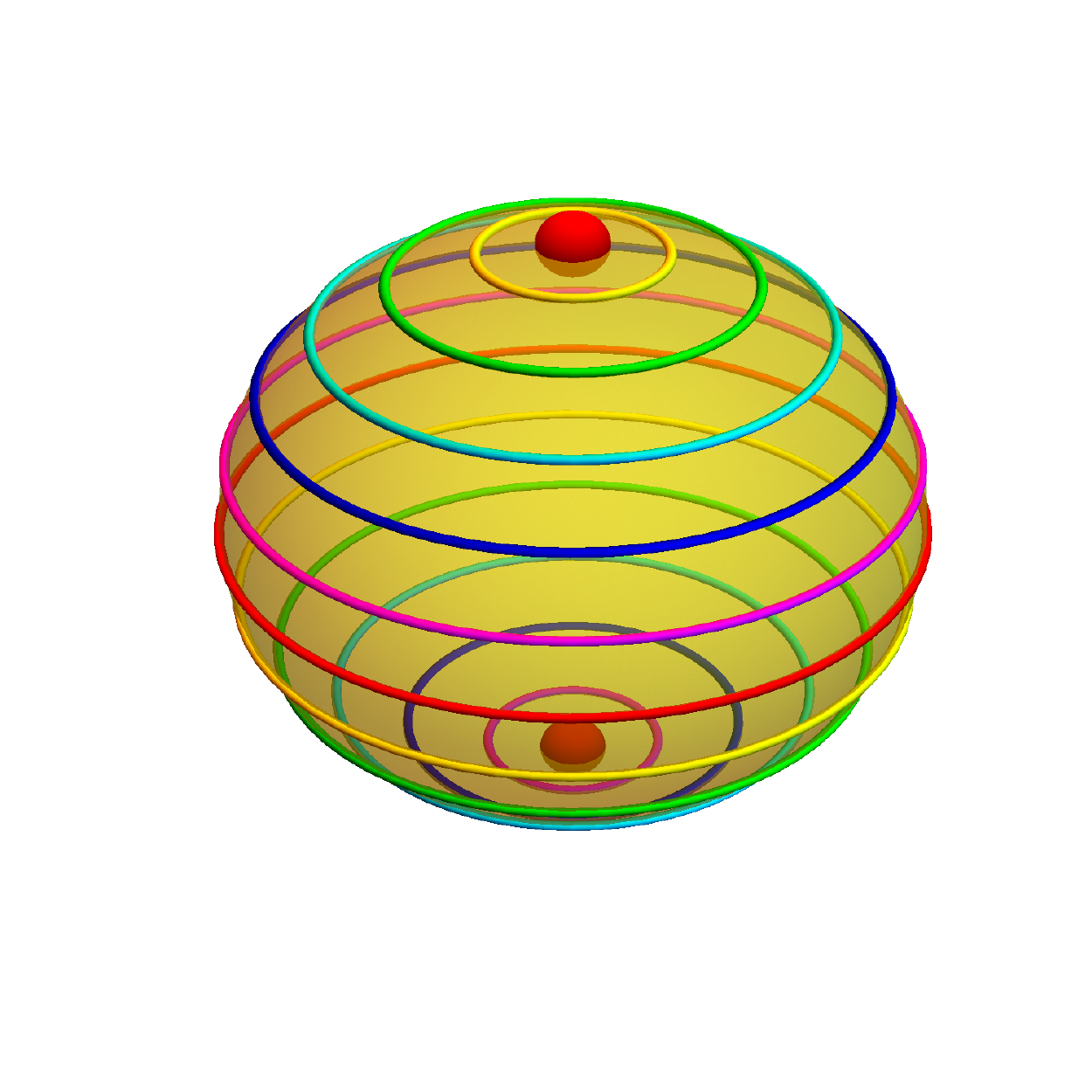}}
\label{reeb}
\caption{
The discrete Reeb theorem for graphs characterizes $d$-spheres as
$d$-graphs admitting a function with exactly two critical points. 
A $d$-sphere can be foliated into $(d-1)$-spheres together with the 
critical points. 
}
\end{figure}

\paragraph{}
We will study the variational problem for the functional $\chi$ more elsewhere and especially
look for conditions which are needed for $\chi$
to be extremal. One condition is that the Green function $g(x,x)=L^{-1}(x,x) = 1-\chi(S(x))$ entries have
a definite sign. This is satisfied for $d$-graphs where all unit spheres $S(x)$ are $(d-1)$ spheres of
the same dimension. But $\chi$ has extrema also some varieties as there are examples like Bouquets
of 2-spheres, where critical points have Green function entries with definite sign.

\begin{figure}[!htpb]
\scalebox{0.57}{\includegraphics{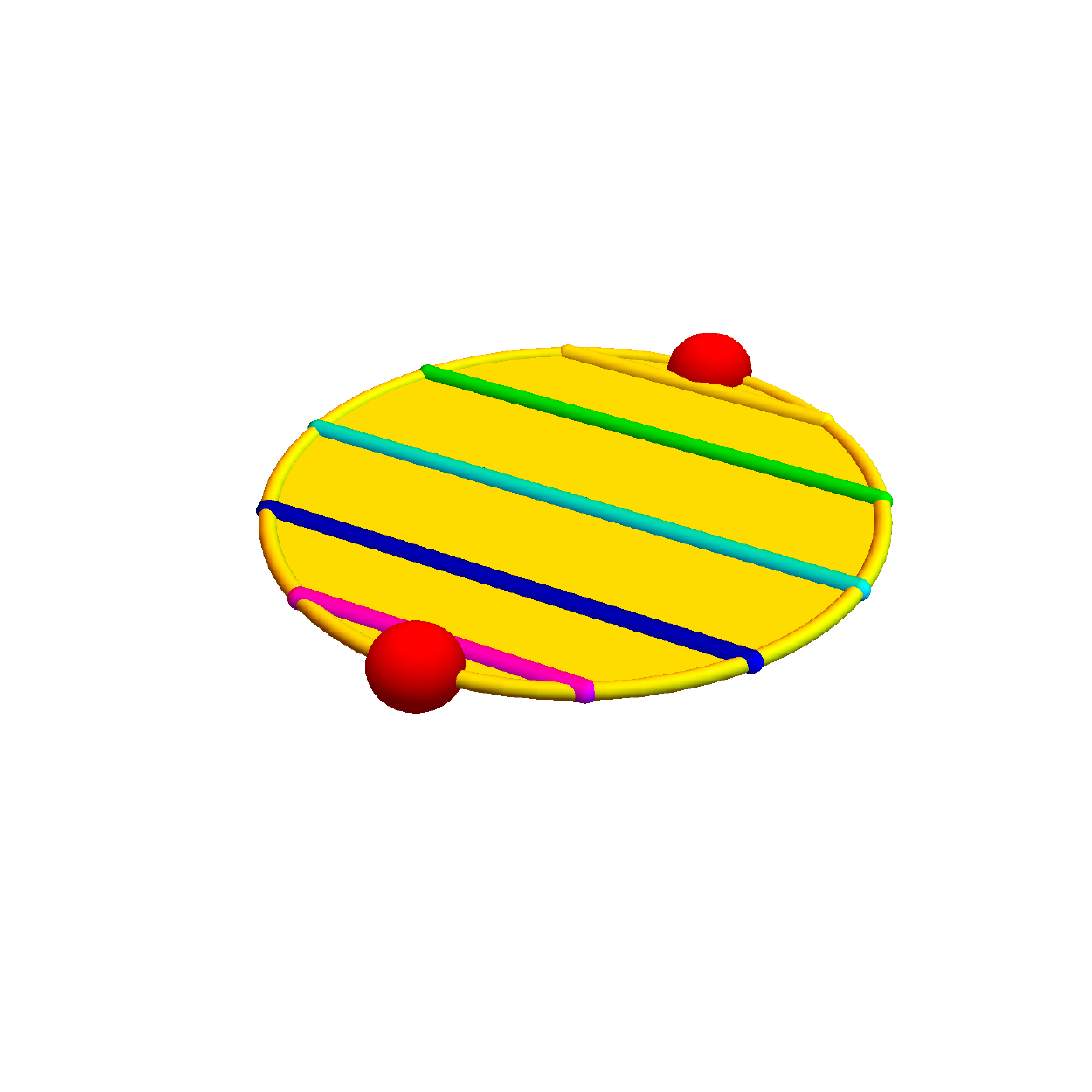}}
\scalebox{0.57}{\includegraphics{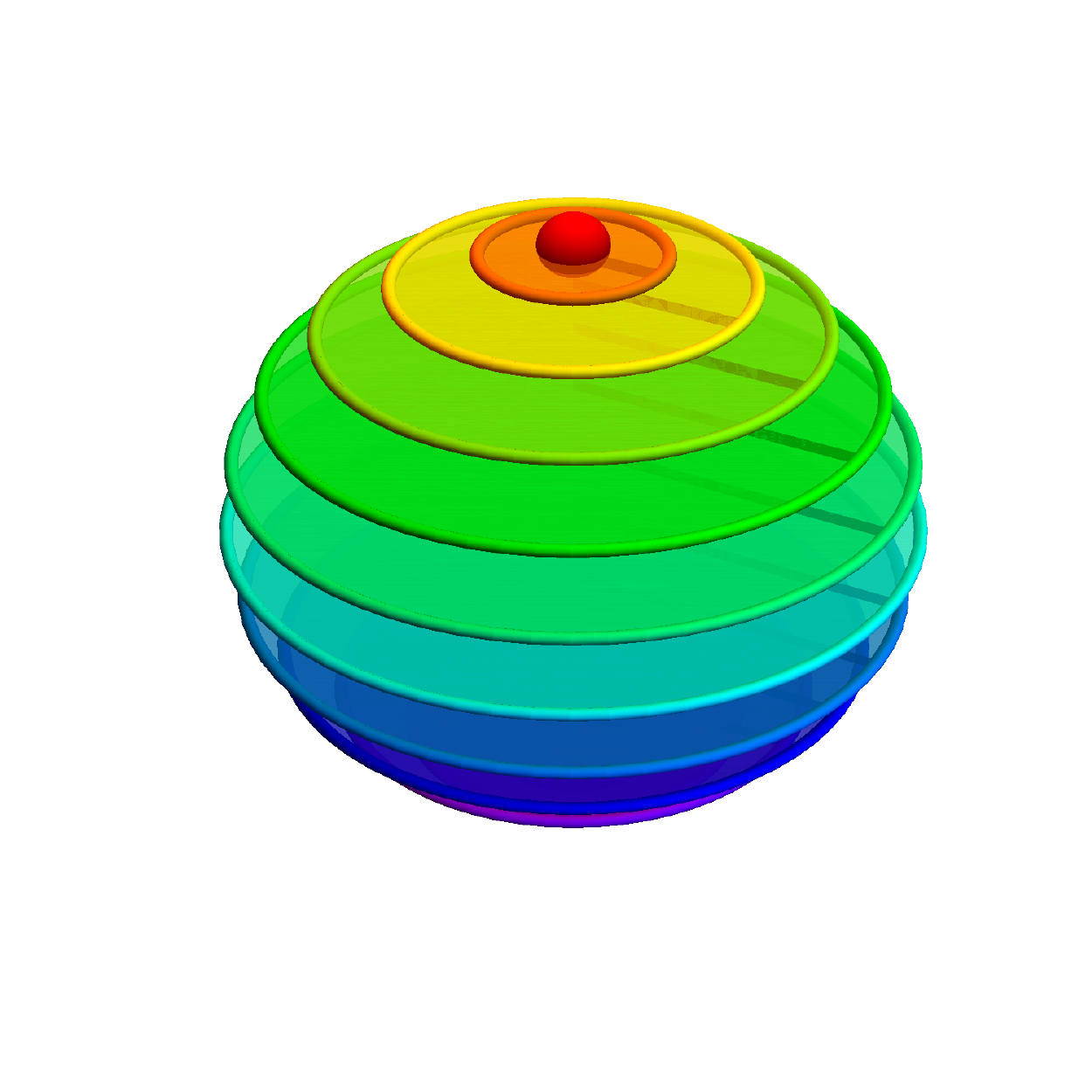}}
\label{reeb}
\caption{
The Reeb theorem for $d$-graphs with boundary 
characterizes $d$-balls as $d$-graphs with boundary 
which admit a function with exactly two critical points. 
A $d$-ball can be foliated into $(d-1)$-balls closed of
by the two critical points.
}
\end{figure}

\section{Critical points}

\paragraph{}
Given an arbitrary finite simple graph $G=(V,E)$ and a locally injective function $f:V \to \mathbb{R}$,
we have at every point $x \in V$ a graph $B_f(x) = \{ f=f(x) \} \subset S(x)'$ (\cite{KnillSard})
which is generated by the simplices in the Barycentric refinement $S(x)'$ of $S(x)$, 
on which $f-f(x)$ takes both positive and negative signs. This {\bf center manifold} $B_f(x)$
divides the sphere $S(x)$ into the two sets $S_f^-(x) =  \{ f(y)<f(x) \}$
and $S_f^+(x)=\{ y \in S(x) \; | \; f>f(x) \}$. We call $x$ a {\bf critical point} of $f$ if
one or both of the sets $S_f^{\pm}(x)$ is not contractible. For a regular point
the {\bf Poincar\'e-Hopf indices} $i_f^{\pm}(x) = 1-\chi(S_f^{\pm}(x))$ are zero 
but having a zero index does not guarantee dealing with a regular point.

\paragraph{}
The {\bf symmetric index} $j_f(x)=(i_f(x)+i_{-f}(x))/2$ can for
odd $d$ be written as $j_f(x)=-\chi(B_f(x))/2$ which is zero because the 
Euler characteristic of an odd-dimensional d-graph is zero and $B_f(x)$ is a $(d-2)$-dimensional graph.
The symmetric version of the usual index $i_f(x)$ \cite{poincarehopf} has appeared in
\cite{indexformula} and the definition works for general finite simple graphs. We look at it here
for $d$-graphs, which are locally Euclidean graphs, discrete versions of manifolds.

\paragraph{}
By definition, a contractible graph is characterized by the fact that there is a function $f$
with exactly one Poincar\'e-Hopf critical point, a point where $S_f^-(x)$ is not contractible.
The index $i_f(x) = 1-\chi(S_f^-(x))$ can still be zero for such a critical point.
With the more symmetric version of critical point, asking that either $S_f^-(x)$ or
$S_f^+(x)$ is not contractile, a graph always admits at least two critical points,
the maximum and minimum. From the Lusternik-Schnirelmann point of view \cite{josellisknill}, 
the Poincar\'e-Hopf critical point is more reasonable as it allows to give sharper inequalities 
which match the inequalities in the continuum: cup-length is a lower and the minimal number of
critical points an upper bound for the Lusternik-Schnirelmann-category. 

\paragraph{}
When using the notion of $d$-graph which is inductively defined through the property of having
unit spheres which are spheres, the concept of a ``level surface" is particularly nice. For any value $c$
different from the range of a locally injective function $f$, the graph $\{ f = c \}$ is always
a $(d-1)$-graph. This can be rephrased that all topology changes happen at the vertices of the
graph. Unlike in algebraic geometry, where one has to deal with singularities
of varieties given by the zero locus of a finite set of polynomials, we do not have to worry
about singularities away from the range of $f$ which in the finite case is a finite set.
New in the ``discrete Sard" result \cite{KnillSard} is that there is non-commutativity 
in that $\{ f=c, g=d \}$ first has to build $\{ f=c \}$, then get $g$ on that surface.
Building $\{ g=d \}$ and there look at $f=c$ is in general different. 

\paragraph{}
We will see here that for $d$-graphs which are by definition graphs without boundary,
the notion of {\bf Poincar\'e-Hopf critical point} and {\bf symmetric critical point}
(asking that at least one of the two graphs $S_f^{\pm}(x)$ are not contractible) are the same.
It allows to give a sphere-based definition what a Morse function is.
In the discrete, we don't have derivatives and Hessian matrices, not even notions like straight
lines or geodesics so that all geometry has to be defined within a type of 
{\bf sphere geometry}.

\section{The Reeb theorem}

\paragraph{}
Classically, the Reeb sphere theorem from 1952 \cite{Reeb1952}  (referred to and generalized in
\cite{McAuley} who generalizes and improves on results by Milnor and Rosen) 
shows that in the class of closed compact smooth manifolds, the ones which admit a function 
with exactly two critical points must be sphere. The level curves
of such a function then defines a foliation of the sphere for which the leaves
are smaller dimensional spheres or then degenerate to points at the two critical points. 
The following discrete version is commented on more in the last section. 

\begin{thm}[Discrete Reeb]
For $d \geq 0$, a $d$-graph $G$ admits a function with exactly two critical points 
if and only if it is a $d$-sphere.
\end{thm}
\begin{proof}
a) If a graph $G=G_0$ is a $d$-sphere, there exists $x_0$ such that $G_1=G-x_0$ is contractible. 
Define $f(x_0)=0$. As $G_1$ is contractible, there exists $x_1$ such that 
$G_2=G_1-x_1$ is contractible and $S(x_1)$ is contractible. Define $f(x_1)=1$. 
In $G$, we have $S^-_f(x_1)$ and $S^+_f(x_1) = \{ x_0 \}$ are contractible and
$B_f(x_1)=S(x_0) \cap S(x_1)$ is a sphere. Continue like that to see that
all points are regular points except for the last point $x_{n-1}$, where we declare 
$f(x_{n-1})=n-1$, and where $S_-(x)$ is empty. The function $f$ has exactly $2$ 
critical points $x_0$ and $x_{n-1}$.  \\
b) Assume now that $G=(V,E)$ is a $d$-graph and that a function $f: V \to \mathbb{R}$ is given 
with exactly two critical points. Define $G_0=G$. Take the first critical point $x_0$ of $f$ to get
$G+1=G-x_0$. Now we can take $x_1$ away from $G_1$ to get $G_2=G_1-x_1$ which 
is contractible. Continue like this to get a sequence of graphs $G_k$ for which 
$G_{k+1}=G_k-x_k$. As each $x_k$ with $k=2, \dots, n-2$ is  not a critical point, every 
$G_k$ is contractible. Eventually, the graph $G_{n-1}$ is a one-point graph containing one
point $x_{n-1}$. As $x_{n-2}$ was not a critical point, $G_{n-1}=G_n +_{S^-(x_{n-1})} x_{n-1}$ 
is contractible etc so that $G_1$ is contractible which by definition checks that $G=G_0$ 
is a $d$-sphere. 
\end{proof}

\paragraph{}
The following lemma shows that every contractible sub-graph $K$
of a $d$-graph $G$ defines a $d$-ball in the Barycentric refinement.
Figure~\ref{contractibletoball} illustrates this. The graph $K$
can be rather arbitrary, as long as it is contractible. 
It can be a tree for example. The ball $B$ defined by it 
is the union $B=\sum_{x \in V(K)} B(\{x\})$, where $B(\{x\})$ is the
unit ball of the vertex $\{x\}$ in the Barycentric refinement $G'$ of $G$.
The union of these balls is a ball. We formulate it using a function $f$
which is negative on $K$ and positive outside $K$. The graph
$\{f \leq c\}$ is the set of simplices $x$ for which $f$ is either constant
negative on $x=\{x_1, \dots x_k\}$ 
or then takes both positive and negative values on $x$.

\begin{lemma}
Given a $d$-graph $G$ and a contractible sub-graph $K$.
Define a function $f$ which is negative on $K$ and positive everywhere else.
Then, the surface $\{ f=c \}$ is a sphere bounding the ball 
$B=\{ f \leq c \}$ in $G'$.
\end{lemma}

\begin{proof}
We use induction with respect to the number $n$ of vertices of $K$. If $K$ has one point,
then it is a unit ball $B(x)$ and the surface $\{ f = c \}$ is a sphere (it is not
graph theoretically isomorphic to $S(x)$, but one can define an equivalent 
discrete cell complex on it which renders it equivalent).\\
If $K=\{x_1, \dots, x_n\}$ is given in the order with which the
contractible set $K$ is built up, start with $B(x_n)$ and
define $U=B(x_n) \cap \bigcup_{k=1}^{n-1} B_k$.
By induction assumption $\bigcup_{k=1}^{n-1} B(x_k)$ is a ball and
$B(x_n)$ is a ball. 
In general, in the Barycentric refinement $G'$ of $G$, 
the union of a ball $U = \bigcup_{k=1}^{n-1} B(x_k)$, (where the $x_k$ are all 
zero dimensional in the original graph $G$) and a unit ball $B(x_n)$ of a 
boundary point $x_n$ of $U$ (which is also a zero dimensional simplex in $G$) 
is still a ball: in order to verify this, we only have to check that
for a point $y$ in the intersection of the boundaries of $U$
and $B(x_n)$, the unit sphere is a $(d-1)$ ball.
\end{proof}

{\bf Remark.} The statement ``if $U \subset G$ is a ball in a $d$-graph $G$
and $x$ is a boundary point of $U$, then $U \cup B(x)$ is a ball in $G$" is
wrong in general. In general, there are refinements needed which make sure that
the new ball $B(x)$ does not touch $U$ elsewhere. 

\begin{thm}[Foliation]
A $d$-graph $G$ is a $d$-sphere if and only if it admits a 
function $f$ such that for every $c \notin {\rm im}(f)$, 
the graph $\{ f = c \}$ is either empty or a $(d-1)$-sphere. 
\end{thm}
\begin{proof}
Both directions follow from the just proven theorem but we need the above lemma. 
a) Assume $G$ is a $d$-sphere. By the just proven Reeb theorem,
there is a function $f$ with exactly $2$ critical points. 
For any $c$ in the complement of the range of $f$, 
the graph $S=\{ f = c \}$ is a $(d-1)$-sphere. The reason is that
$f<c$ and $f>c$ are both contractible and so balls
with the boundary $S$ as a $(d-1)$-sphere.  \\
b) Let now $G$ be a $d$-graph for which a function $f$ exists with the assigned
properties. This function has exactly two critical points. The reason is that 
$ \{ f = c \} \cap S(x)$ is $B_f(x)$ which is a sphere, implying that the point is
a regular point. 
\end{proof} 

\paragraph{}
If $x$ is a regular point of $f$, then both $S^+_f(x)$ and $S^-_f(x)$ are 
$(d-1)$-balls and $B_f(x)$ is a $(d-2)$-sphere:

\begin{coro}[Regular center manifold]
For any $d$-graph and a regular point $x$ for $f$, the center manifold 
$B_f(x)=\{ f = f(x) \}$ inside the $(d-1)$-sphere $S(x)$ is always a $(d-2)$-sphere. 
\end{coro}
\begin{proof}
As one side is contractible, the lemma above shows that it is a ball. Its
boundary is then a sphere. It is $B_f(x)$ in the Barycentric refinement. 
\end{proof}

\paragraph{}
The reverse is a bit more difficult. But here is the ball version:

\begin{thm}[Balls]
Any $d$-ball admits a function $f$ with exactly two critical points. 
For such a function $f$, every level surface $\{f = c\}$ is either empty or a 
$(d-1)$-ball. 
\end{thm}
\begin{proof}
The proof of the Reeb theorem shows that any pair 
of two points can be chosen to be critical points (maxima or minima). 
Start by applying the Reeb theorem to the boundary sphere $S$ of $B$. 
Let $\{a,b\}$ be the critical points. 
Complete $G$ by making a cone extension over the boundary. With the added point, it
becomes a sphere $G+x_0$. Again by Reeb, there is a function on the completed sphere which has
the two critical points $a,b$. Every 
intersection of $f=c$ with $B(x_0)$ is now either empty or 
a $(d-1)$-ball. This is a ball foliation we were looking for. 
\end{proof} 

\paragraph{}
In general, at critical points, the center manifold $B_f(x) = \{ f = f(x)\}$ 
can be rather arbitrary. Any manifold which can occur as a hypersurface of a 
$(d-1)$-dimensional Euclidean sphere can also occur as a surface $B_f(x)$ and
the same is true in the discrete. To realize it, we would just have to make sufficient
many Barycentric refinements first. 

\paragraph{}
Here is the equivalence of the {\bf symmetric notion of contractibility} and
{\bf one sided notion of contractibility} in the class of $d$-graphs.

\begin{thm}[Critical points]
Given a function $f$ on a $d$-graph. Then for every unit sphere $S(x)$, the 
subgraph $S^-_f(x)$ is contractible if and only if $S^+_f(x)$ is contractible.
\end{thm}
\begin{proof}
We know from the lemma that $\{ f \leq c\}$ is a ball and $S$ is a $(d-1)$-sphere.
Because $K$ is contractible, we can reduce $K$ to less and less
vertices and still keep $G$ a sphere. Once $K$ is a 1-point graph $x_1$, 
then $S$ is the unit sphere $S(x_0)$ and the punctured sphere 
$S - x_0 = S^+_f(x)$ is a $d-1$ ball and so contractible. 
\end{proof}

\paragraph{}
It follows that if $S=\{ f=c \} \subset G'$ in a $d$-sphere in a graph $G$ for which 
$K=\{ y \in V(G) | f(y) < c \}$ is contractible, then $S$ is a $(d-1)$-sphere and 
both $\{ f \leq c \} \subset G'$ and $\{ f \geq c \} \subset G'$ are $(d-1)$-balls. 

\begin{figure}[!htpb]
\scalebox{0.95}{\includegraphics{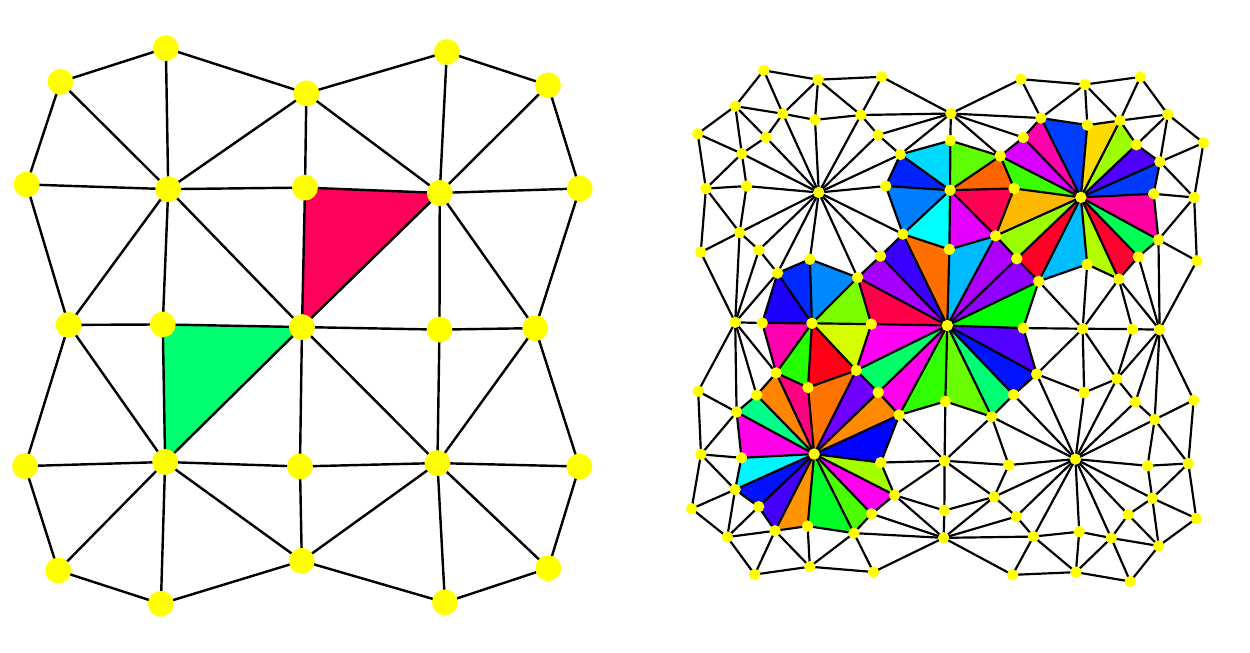}}
\label{contractibletoball}
\caption{
The picture shows a contractible graph $K$ inside a $2$-ball $G$. It has 5 vertices and is not a
ball. It produces a $2$-ball $\{ f \leq c \}$ in the Barycentric refinement $G'$ which has 
a $1$-sphere (a circular graph) as a boundary.
}
\end{figure}

\section{Morse functions} 

\paragraph{}
We call a locally injective function $f$ a {\bf Morse function}, if at every critical 
point $x$ of $f$, the center manifold $B_f(x)$ is either empty or a {\bf Cartesian product} of two spheres. 
At a regular point we by definition have that $B_f(x)$ is a sphere. What is different in 
the Morse case that the structure of the center manifold is assumed to be special. In general, 
the center manifold $B_f(x)$ can be a quite arbitrary hypersurface in $S(x)$. 

\paragraph{}
For a Morse function $f$ and a regular point $x$ of $f$,
the unit sphere $S(x)$ is homeomorphic to the suspension of the center manifold 
$B_f(x)$, where $B_f(x)$ is a $(d-2)$-sphere. At a 
critical point $x$, the unit sphere $S(x)$ is homeomorphic to the join 
$S^k + S^l$, where $B_f=S^k \times S^l$ with $k+l=d-2$. This makes sense because
the join has one-dimension more than the Cartesian product and additionally is 
a sphere, while the Cartesian product is never a sphere. 

\paragraph{}
The situation mirrors the continuum classical differential topology case, where at a Morse critical point,
a small sphere $S(x)$ is the join of the {\bf stable unit sphere manifold} $S(x) \cap W^-(x)$ and 
{\bf unstable unit sphere manifold} $S(x) \cap W^+(x)$, where $W^{\pm}(x)$ are the 
standard {\bf stable and unstable manifolds} of the gradient field $\nabla f$ at $x$. The center
manifold in the continuum is the Cartesian product of $S(x) \cap W^-(x)$ and $S(x) \cap W^+(x)$. 
See Figure~(\ref{hyperbolic}). 

\paragraph{}
The classical case follows from the Morse lemma. In a coordinate
system near a critical point, in which $f=x_1^2+ \cdots + x_k^2 - x_{k+1}^2 - \cdots - x_d^2$ 
then $f=0$ intersected with a small sphere
$S_r(0) = \sum_j x_j^2 =r^2$ gives either $B_f(0)=\emptyset$ (at maxima $k=d$ or minima $k=0$) 
or then the product 
$$  \{ x_1^2+ \cdots + x_k^2 = r^2/2 \} \times \{ x_{k+1}^2+ \cdots + x_d^2 = r^2/2 \}  \; . $$
Of course, at regular points of a smooth function, 
$B_f(x) = S_r(x) \cap \{ f = f(x) \}$ is a $(d-1)$-sphere for small enough $r$. 

\paragraph{}
For a critical point with
Morse index $1$ or $2$ in a $3$-manifold for example, the center manifold 
$B_f(x) =\{ y \in S(x) \; | \; f(y) = f(x) \; \}$ is the intersection of a cone with a sphere which consists 
of two circles, which is $S^1 \times S^0$. At a maximum or minimum (where the Morse index is $0$ or $3$), the
center manifold is empty, which is the Cartesian product of $S^2 \times S^{-1}$. At a regular point,
the center manifold $B_f(x)$ is always a $1$-sphere, the intersection of the level surface 
$f=f(x)$ with $S(x)$.  \\

The formula \cite{indexformula} 
$$  j_f(x)=1-\chi(S(x))/2-\chi(B_f(x))/2 = 1-(\chi(S^+(x) +\chi(S^-(x)))/2 $$
holds at every vertex $x$ of an arbitrary finite simple graph. In the case of $d$-graphs with 
even $d$, it simplifies to $1-\chi(B_f(x))/2$ which leads to a poetic interpretation
of curvature at $x$ as an expectation of Euler characteristics of random $(d-2)$-graphs $B_f(x)$ 
in $S(x)$ \cite{indexexpectation,indexformula,eveneuler}. This bootstraps with Gauss-Bonnet to the fact 
that curvature in even dimension always an expectation of Euler characteristic of a well defined probability space
of random 2-graphs. If $d$ is odd, then it shows that curvature
is an expectation of Euler characteristic of $(d-2)$-graphs which bootstraps to the statement
that odd-dimensional manifolds have zero curvature everywhere and so by Gauss-Bonnet 
zero Euler characteristic. 

\paragraph{}
The following result uses in one direction the deeper Jordan-Brouwer-Schoenfliess theorem:

\begin{thm}[Regular points have spheres as center manifolds]
For a $d$-graph $G$, the center manifold $B_f(x)$ is a $(d-2)$-sphere if and only 
if $S_f^-(x)$ and $S_f^+(x)$ are both contractible. 
\end{thm}
\begin{proof}
(i) If $B_f(x)$ is a $d-2$ sphere, then by the Jordan-Brouwer-Schoenfliess theorem, it
divides the $(d-1)$-sphere $S(x)$ into two parts which are both balls. 
Consequently they are contractible. \\
(ii) If one of the $S_f^-(x)$ and $S_f^+(x)$ is contractible, then their Barycentric
refinements are balls complementing a $(d-2)$-graph. Thus $B_f(x)$ must be a sphere.
\end{proof}

\section{Remarks}

\paragraph{}
Weyl in \cite{Weyl1925} (page 10) writes: 
{\it "Es ist zu fordern, dass diejenigen Elemente niederer Stufe, welche ein Element n'ter
Stufe begrenzen, ein m\"ogliches Teilungsschema nicht einer beliebigen (n-1)-dimensional Manifaltigkeit
sondern insbesondere einer (n-1) dimensional Kughel im n-dimensional Euklidischen Raum bilden. Und es ist
bisher nicht gelungen fuer $n \geq 4$ die kombinatorischen Bedingungen daf\"ur zu ermitteln."}
Weyl essentially describes simplicial complexes and in the following then describes how to get to the
continuum using Barycentric refinement. In the above sentence, he essentially states that a discrete
manifold should have a neighborhoods which are (n-1)-dimensional spheres which leads to the problem to 
describe spheres combinatorically. 

\paragraph{}
Weyl mentions taking an ``arbitrary (n-1)-dimensional manifold" as a unit 
sphere, one can look for example at the join $G=\mathbb{T}^2 \oplus K_1$ which is a contractible 3-manifold
in which the unit sphere of one point is the 2-torus. Since the space is three-dimensional and contractible
all topological notions (like homotopy or cohomology groups) are the same than for a 3-ball. The Euler
characteristic of course is $1$, the Betti numbers are $(1,0,0,0)$ like for the 3-ball. However,
the Wu characteristic of $\mathbb{T}^2 \oplus K_1$ is $1$, while the Wu characteristic of a 3-ball is $-1$. 
Also, any function on $G$ has more than 2 critical points. For example, if the maximum is at the top point $p$ 
of the pyramid, then there is a minimum on the base (the torus $\mathbb{T}^2$). But since 
$i_f(p)=1-\chi(S_f(p))=1-\chi(\mathbb{T}^2)=1$ and at the minimum $i_f(p)=1-0=1$ and the Euler characteristic is
$1$, there has to be an other point. 


\begin{figure}[!htpb]
\scalebox{0.27}{\includegraphics{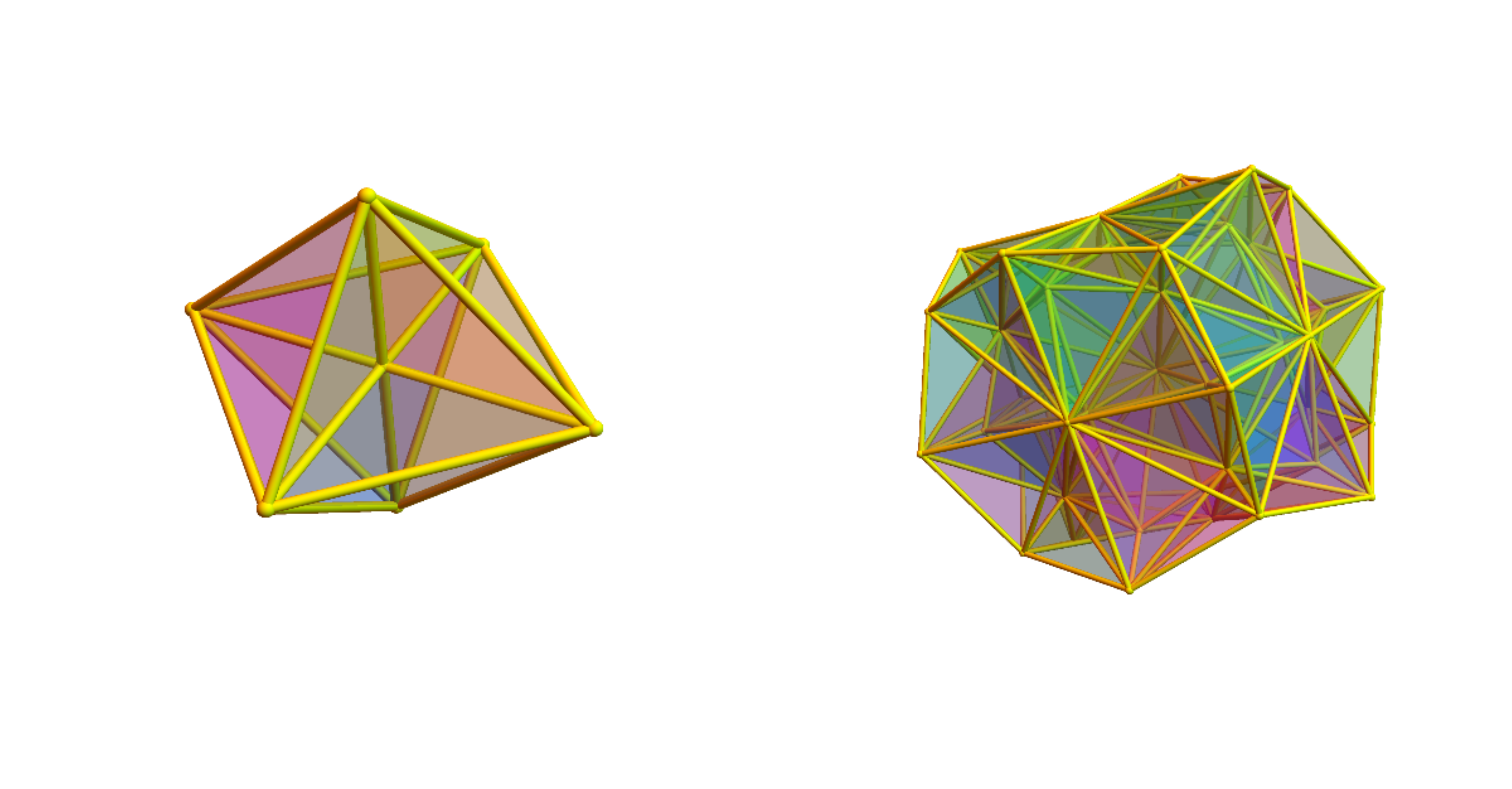}}
\scalebox{0.14}{\includegraphics{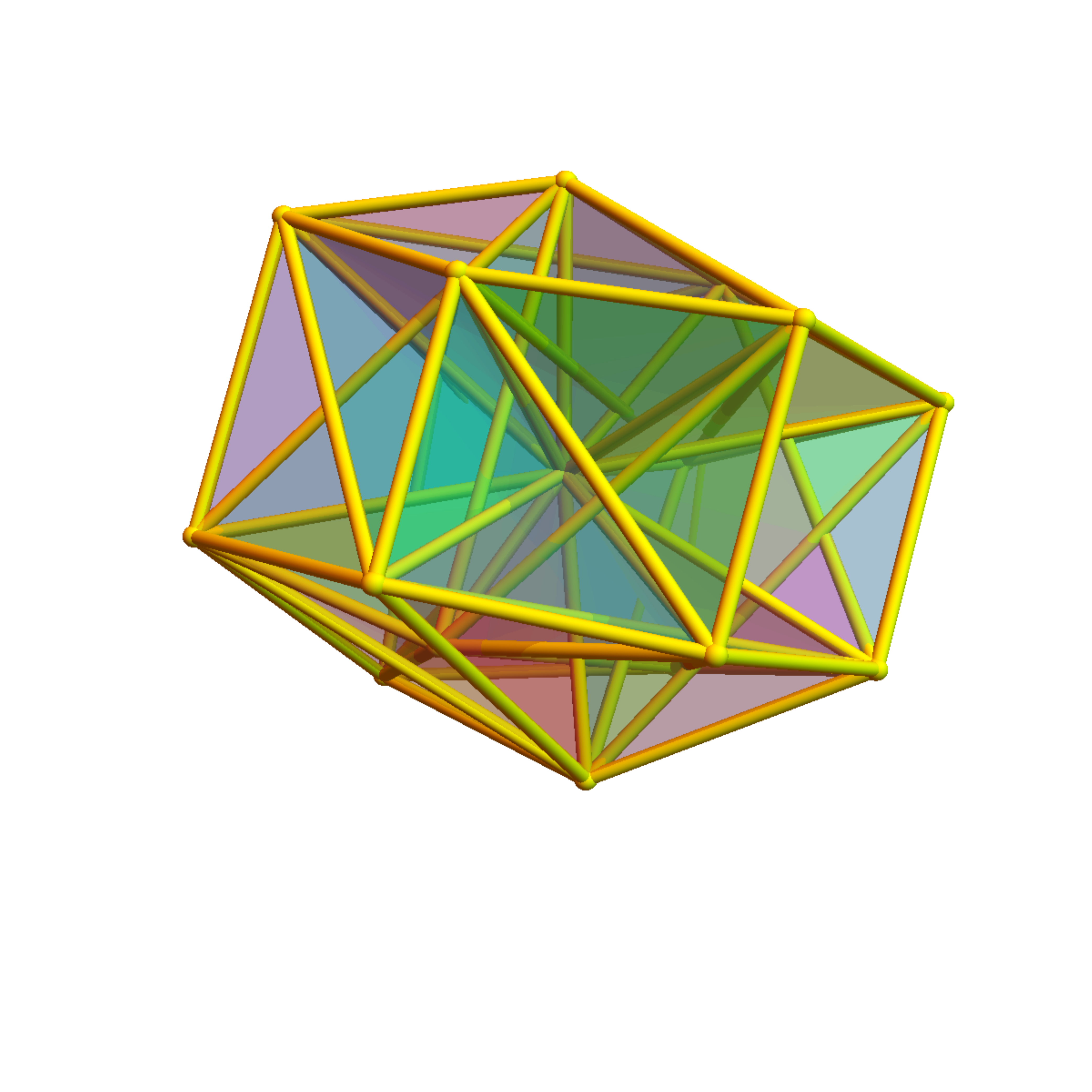}}
\label{reeb}
\caption{
A three ball $S^2 \oplus K_1$, where $S^2$ is the octahedron graph and its
Barycentric refinement. Below, we see $T^2 \oplus K_1$, where $T^2$ is a
discrete 2-torus. It is not a 3-ball as one point has a torus as a unit sphere. 
But it is contractible. Only if we look at the topology of the boundary, we can see
it. The Wu characteristic is $\omega(G)=\chi(G)-\chi(\delta G)=1-0=0$. For the 3-ball
$\omega(G)=\chi(G)-\chi(\delta G)=1-2=-1$. 
}
\end{figure}

\paragraph{}
The classical Reeb sphere theorem assures that if a differentiable $d$-manifold admits a differentiable
function with exactly two critical points, then it is homeomorphic to the standard Euclidean
$d$-sphere (not necessarily diffeomorphic although). (For a proof see \cite{Mil63} Theorem 4.1) 
The combinatorial analog of Reeb's theorem is an if and only if statement. 
This is absent in the continuum.  The statement 
``M is a smooth manifold homeomorphic to a sphere, then $M$ admits a function with exactly 
two critical points" is only true $d \leq 6$ (see \cite{Matsumoto} Theorem 3.6). 
This can not be improved because the exotic $7$-spheres of Milnor 
admit a function with exactly two critical points. The Reeb theorem was used by Milnor to establish that 
they are homeomorphic to standard spheres. 

\paragraph{}
We have defined $f$ to be Morse, if every $B_f(x)$ is either a sphere or the product of 
two spheres $S_k \times S_l$ with $k+l=d-2$ or then the empty graph. 
In the case $d=2$, we either have $S^0 \times S^0$ (hyperbolic points) 
with $j(x)=-1$ or the empty graph $0$ (maxima or minima) with $j(x)=1$. 
In the case $d=3$, we either have $S^1 \times S^0$ (Morse index 1 or 2)
or the empty graph $0$ (maxima or minima, Morse index 3 or 0) and the index $j(x)$ is always zero. 
For $d=4$, the center manifold can be $S^1 \times S^1$ (Reeb 2-torus in a 3-sphere $S(x)$) or 
$S^0 \times S^2$ or $0$. The index can be negative for $S^0 \times S^2$ only. 
Otherwise it can be $1$.  Having the center manifold to be connected (a $2$-torus in $d=4$ 
is certainly a new phenomenon which starts to appear in dimension $4$. For non-Morse function, 
still in 4 dimension, the center manifold can be a pretty arbitrary orientable 2-surface.

\begin{figure}[!htpb]
\scalebox{0.57}{\includegraphics{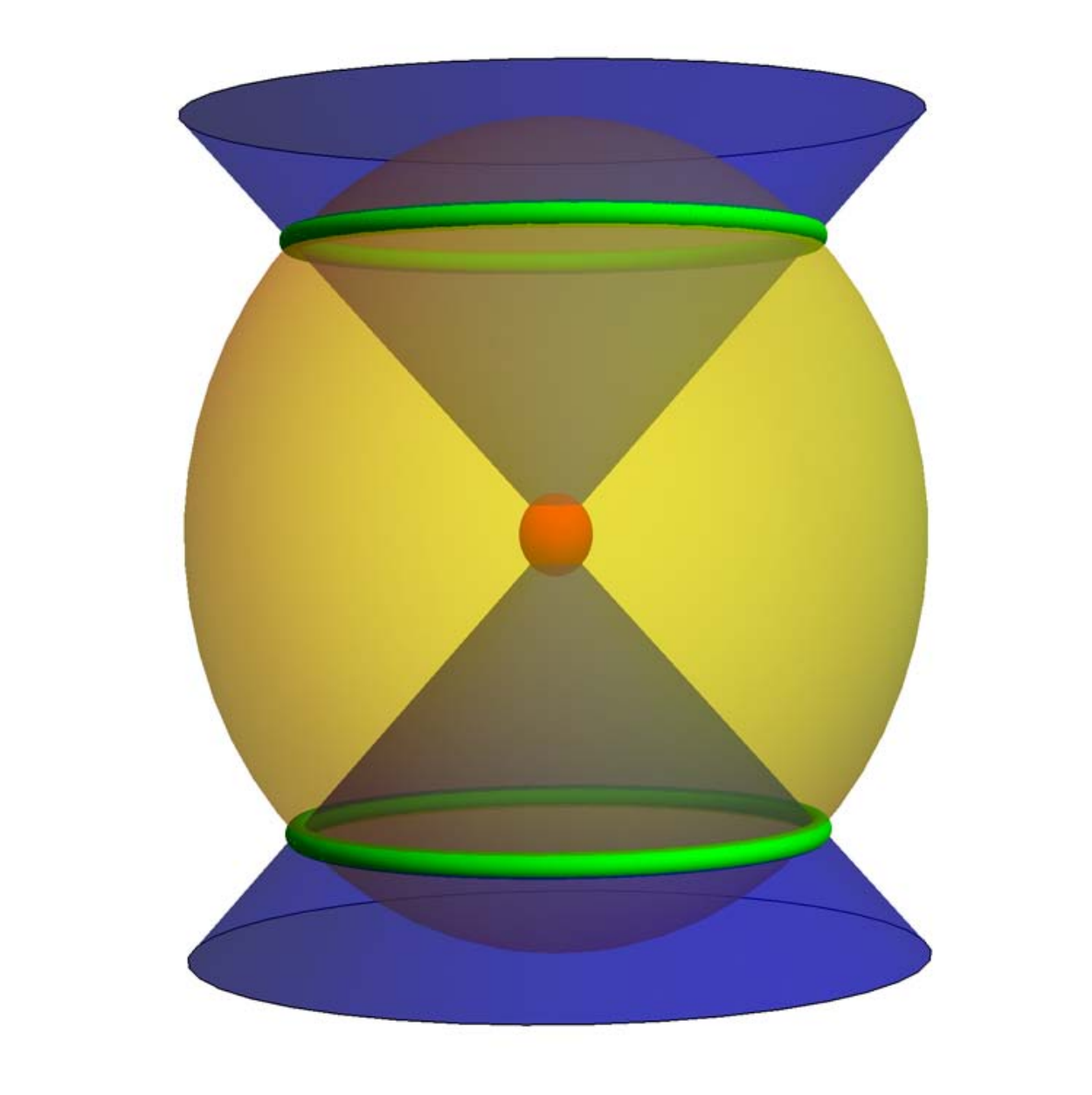}}
\label{hyperbolic}
\caption{
The center manifold $B_f(x)$ of the function $f(x,y,z)=x^2+y^2-z^2$ in $\mathbb{R}^3$ is
the intersection of the cone $\{ f = f(0,0,0) \}$ with a small sphere $S_r(x)$. In the 
Morse case, this is the product $S^0 \times S^1$ of two spheres. 
}
\end{figure}

\paragraph{}
The Reeb theorem implies that every $d$-sphere admits a Morse function. We have not
proven yet that any $d$-graph admits a Morse function. Actually, it would not surprise if this
would turn out to be false for some d-graphs $G$. The fact that 
Morse functions exist after a few Barycentric refinements $G$ follow from the 
classical result that there, Morse functions are generic. 
But also this result on Barycentric refinement should first be proven within 
a discrete frame work. 

\paragraph{}
One can call the Cartesian product of two spheres a {\bf generalized torus}.
In general, the center manifold can be an arbitrary $(d-2)$-graph, as long as it
can be embedded as a level surface $\{ f=c \}$ in the $(d-1)$-sphere $S(x)$.
To realize a particular manifold, we can build a sufficiently fine triangulation of
a $(d-1)$ ball which is embeddable into $\mathbb{R}^{d-1}$ then take the function $f$
and build the graph.

\paragraph{}
The Reeb sphere theorem already comes with some subtlety in the continuum: 
if a differentiable $d$-manifold admits
a smooth function with exactly two non-degenerate critical points
then it is {\bf homeomorphic} to a $d$-sphere by the Reeb sphere theorem.
But it is not necessarily {\bf diffeomorphic} to a $d$-sphere as exotic sphere 
constructions shows. It is actually vexing that we {\bf do not seem to know of any 
differentiable manifold homeomorphic to a $d$-sphere for which the minimal number
of critical points is larger than two.} The smooth Poincar\'e conjecture would 
imply that none exist in dimension 4. The Milnor examples have two critical points. 

\paragraph{}
Also some kind of converse to Reeb is true (although trivial): 
a compact differentiable manifold {\bf diffeomorphic} to a sphere admits
such a function but this is not a converse as that would require the assumption to be 
only homeomorphic, which is not true. Even a sphere homeomorphic but not diffeomorphic to 
the standard sphere can admit such a Morse function (as Milnor has shown in 1956). So, the admission
of a Morse function with this property does not imply the manifold to be diffeomorphic to the 
standard sphere. The story in the continuum is still not settled as one does not know for example
whether exotic spheres in dimension 4 exist and the {\bf smooth Poincar\'e conjecture} claims 
that none do exist. Also no exotic spheres in dimension 5 or 6 are known. 

\paragraph{}
For a $d$-graph and a regular point, $B_f(x)$ always is a $(d-2)$-sphere.
For odd $d$ and a $d$-graph the symmetric index is $-\chi(B_f(x))$ and always zero,
immediately establishing again that the Euler characteristic of such graphs is zero.
Also other Dehn-Sommerville relations immediately can be bootstrapped 
from lower dimension to higher dimensions.
For $4$-graphs, the center manifolds $B_f(x)$ are always $2$-graphs,
discrete two-dimensional surfaces. Any probability
measure on locally injective functions produces a curvature $\kappa = {\rm E}[j_f]$, curvature at a point
is an expectation $1-\chi(B_f(x))/2$ of random surfaces $B_f(x)$. For $4$-graphs in particular
we have at every point a collection of random $2$-graphs $B_f(x)$ and the expectation of $1-\chi(B_f(x))/2$
is the curvature at the point. At a positive curvature point, there are positive genus surfaces $B_f(x)$
while at a negative curvature point, $B_f(x)$ consists mostly of disconnected $2$-spheres.

\paragraph{}
Here is more about specific lower dimensional cases. 
In the case $d=0$, a $d$-graph $G$ is a discrete set of points. Any function $f$ is
a coloring and every point is a critical point. The theorem tells that $G$ is a sphere
if it contains exactly $2$ points.
In the case $d=1$, a $d$-graph $G$ is a finite collection of circular graphs. Every
function $f$ on $G$ has on each connected component exactly two critical points. In order
to have two critical points we have to have one circular graph.
In the case $d=2$, the condition to have no other critical point than the
maximum or minimum is that at every point, the circular graph $S(x)$ is split by $f=f(x)$
into two parts, one where $f<0$ and one where $f>0$. These two parts are linear path graphs.
At a critical point, the center manifold $B_f(x) = \{ f=f(x)\}$ is either empty, or
then consists of $2k$ points for $k \geq 1$.
In the case $d=3$, critical points can again be maxima or minima
with $B_f(x)=\{ f=f(x) \} = S^2 \times S^{-1}$ being empty (Morse index $0$ or $3$)
or then $\{ f=f(x)\}$ is a union of $2$ or more circles which in the Morse case is
$S^1 \times S^0$ or empty (which is Morse index $1$ or $2$).
In the case $d=4$, at a regular point, the center manifold is $S^2$. At a critical point, 
the center manifold $B_f(x) = \{ f=f(x) \}$ can in principle be any orientable 2-graph 
(non-orientable $2$-graphs like the projective plane or the Klein bottle 
are not embeddable in a 3-sphere).  The type of critical points is determined by the genus
of $B_f(x)$. The only Morse cases are when $B_f(x)$ is empty $0$
(which happens at maxima = Morse index 4 or minima Morse index 0)
or then where $B_f(x)$ is the 2-torus $\mathbb{T}^2 = S^1 \times S^1$ (this is Morse index 2 and can
be seen as a {\bf Reeb torus} in the 3-sphere) or two copies of the $2$-sphere 
$S^2 \times S^0$ or $S^0 \times S^2$ (which are the cases with Morse index $1$ or $3$).

\paragraph{}
Some intuition which comes from the continuum can fail in the discrete:
it is not true in general that if $B$ is a ball in a $d$-graph and $x$ is a boundary
point then $B \cup B(x)$ is a ball. The reason is that there can be edges $(a,b)$
connecting points of $B$ which are not in $B$. Adding a ball $B(x)$ now can produce
bridges, connect different parts of $B$ and changing the topology. The statement is
however true in a Barycentric refinement, if $B$ is a ball which is of the form $f \leq c$.
The reason is that if we have an edge $(a,b)$ for which $f(a)<c, f(b)<c$,
then also the edge is included in $\{ f < c\}$.
However, if $B$ is a $d$-ball in a $d$-graph $G$ and $B = \{ f \leq c \}$ and $x$ is near
a boundary point and $f(x)$ is modified to be $<c$, then $\{ f \leq c\}$ is still a $d$-ball.

\paragraph{}
Related to the previous remark, it is not true that we can build up a d-ball $G$
as a sequence of $d$-balls $G_1=B(x_1) = G_2 - x_2$ etc where $G_{n-1}=G-x_{n-1}$. The simplest example is
a union of two 2-balls $G=B(a) \cup B(b)$ which are glued along a boundary edge. Now $B(a)=G_1$ is a
ball but there is no increasing sequence of balls which lead to $G_n=G$. 
This is related to the fact that for $d \geq 2$ that if $B$ is a $d$-ball different from a unit ball, that we can
take away a vertex and still have a ball. A counter example in the case $d=2$ are two wheel graphs glued
together at a common boundary edge. The set of interior points of $B$ is now disconnected
and taking away any boundary point kills the ball property as the unit sphere from some
boundary points will stop being a $1$-ball as it has become a simplex. 

\paragraph{}
It is also not true that if $B$ is a subgraph of a $d$-graph generated by $\{x \in V \; f(x)<c \}$ 
is contractible that $B$ must be a ball. The interior of a ball does not have to be a ball. 
In any dimension $d$, the graph $B=\{ f<c\}$ can be $0$-dimensional for example or 
be one-dimensional like a tree. But $\{ f \leq c \}$,  the set $V$ of simplices
on which $f<c$ or for which $f$ changes sign equipped with bounds
$E=\{ (a,b) \; a \neq b, a \subset b \; {\rm or} \; b \subset a\}$
is a subgraph of the Barycentric refinement $G'$ of $G$ and is itself a
$d$-graph with boundary.

\paragraph{}
All these difficulties for not-Barycentric refined $d-$graphs 
can be overcome with edge refinements. In the context of the 4-color theorem,
we need to build up a 3-ball along an increasing sequence of 3-balls $B_n$ starting from a unit ball
but we are allowed to do edge-refinements at edges, but only at edges which are
already included in the growing ball $B_n$. 
We can there not just look at the Barycentric refinement. In any case, coloring problems are trivial for 
graphs which are Barycentric refinements $G'$ of a d-graph $G$ as the chromatic number of $G$ is always
$(d+1)$. [The dimension function $f={\rm dim}$ is the minimal coloring and actually is a Morse function for
which the index $i_f(x)$ is $(-1)^{{\rm dim}(x)}$ and for which the Poincar\'e-Hopf theorem establishes
that the Barycentric refinement has the same Euler characteristic.]
So, the construction of a nice Morse filtration $B_k$ with edge refinements done 
is a bit more involved. The Reeb theorem story illustrates this ongoing work.

\bibliographystyle{plain}

\end{document}